\documentclass[12pt,a4paper,reqno,twoside]{amsart}

\usepackage{amsfonts,amssymb,amscd,amsmath,amsthm}
\usepackage{graphicx}
\usepackage{float}
\usepackage{fullpage}
\usepackage{setspace}
\usepackage[utf8]{inputenc}
\usepackage{mathtools}
\usepackage{comment}
\excludecomment{mysection}
\newtheorem{theorem}{Theorem}[section]
\newtheorem{lemma}[theorem]{Lemma}
\newtheorem{proposition}[theorem]{Proposition}
\newtheorem{corollary}[theorem]{Corollary}
\newtheorem{definition}[theorem]{Definition}

\newtheorem{remark}[theorem]{Remark}

\newcommand{\R}{\mathbb{R}}

\newcommand{\T}{\mathcal{T}}
\newcommand{\B}{\mathcal{B}}
\newcommand{\N}{\mathbb{N}}

 \usepackage[bookmarksnumbered, colorlinks, plainpages]{hyperref}
  \hypersetup{colorlinks=true,linkcolor=blue, anchorcolor=green, citecolor=cyan, urlcolor=red, filecolor=magenta, pdftoolbar=true}

\DeclareMathOperator{\cone}{cone}
\DeclareMathOperator{\face}{face}
\DeclareMathOperator{\co}{co}

\DeclareMathOperator{\tr}{tr}
\title{Order ideals in order smooth $p$-normed spaces}
\author{Anindya Ghatak}
\date{}
\begin{document}
\vspace{3in}

\maketitle
\begin{abstract}
We generalize the notion of $M$-ideals in order smooth $\infty$-normed spaces to `` smooth $p$-order ideals''  in order smooth $p$-normed spaces. 
We show that if $V$ is an order smooth $p$-normed space and $W$ is a closed subspace of $V$, then $W$ is a smooth $p$-order ideal in $V$ if and only if $W^{\perp}$ is a smooth $p'$-order ideal
 in order smooth $p'$-normed space if and only if $W^{\perp\perp}$ is a smooth $p$-order ideal in order smooth $p$-normed space $V^{**}$. We prove that every $L$-summand in order smooth $1$-normed 
 space is a smooth $1$-order ideal. We find a condition under which every $M$-ideal in order smooth $\infty$-normed space is a smooth $\infty$-order ideal.
  We show that every $M$-ideal 
 in order smooth $\infty$-normed space is smooth $\infty$-order ideal.
 
 \end{abstract}

\section{Introduction}
 
     In 1972,  E. M. Alfsen with E. G. Effros  introduced the notion of  $M$-ideals for general Banach spaces in a seminal paper \cite{AE72}. The central theme of the paper was the investigation of certain subspaces 
   ($M$-ideals) of $V$ which are analogous to the self-adjoint parts of the closed two sided ideals in $\mathrm{C}^*$-algebras.
     Let $V$ be a Banach space and $W$ be a closed subspace of $V$. Then $W$ is called \emph{$M$-ideal} in $V$ if 
  $$V^*= W^{\perp}\oplus_{1}W^{\perp'}.$$
  
  Let $A(K)$ denote the space of all continuous affine functions on $K$, where $K$ is a compact convex subset of some locally 
convex space $E$  (see e.g.\cite{Alf}). The notion of $A(K)$ spaces was introduced by Kadison in 1951 in order to study the order structure of $\mathrm{C}^*$-algebra \cite{KR51}. He proved that the self-adjoint 
part of a unital $\mathrm{C}^*$-algebra. In particular, and the self adjoint part of unital subspace of it is an order unit space. He further showed that every complete order unit space are isometrically order isomorphic to $A(K)$ space.

 The study of ordered Banach spaces by (different) geometric properties and its ideal theory 
 was initiated in 1950's in the works of And\^o, Bonsall, Edwards, Ellis, Asimov, Ng and many others (see e.g. \cite{A62, ASMW73, ASMW74,BNSL56, BNSL54, ELLIS64}).

  E. Str\o mer \cite{S68} in 1966 studied  Archimedean ideals in order unit spaces ( He called order unit space as Archimedean order unit space) which have a strong order unit and 
  are complete in the order unit norm.
  More preciously, a closed subspace $W$ of $V$ is an \emph{Archimedean ideal} if 
   \begin{enumerate}
    \item [(i)] $W$ is an order ideal in $V$;
      \item [(ii)]$W$ is positively generated;
   \item [(iii)]  $V/W$ is Archimedean.
  \end{enumerate}
  There he characterized that if $\mathcal{I}$ is a closed subspace of a $\mathrm{C}^*$-algebra $\mathcal{A}$, then $\mathcal{I}_{sa}$ is an Archimedean ideal in $\mathcal{A}_{sa}$ if and only if $\mathcal{I}$ is a two sided ideal
  in $\mathcal{A}$.

It follows from \cite{AE72} that every $M$-ideal in $A(K)$-space is an Archimedean ideal. Also an Archimedean ideal $W$ in $A(K)$ is an \emph{M-ideal} if the corresponding (quotient) homomorphism
  $\varphi :A(K)\mapsto A(K)/W$ satisfies: For every $\epsilon >0$ and $a_{1}, a_{2}\in A(K)^+$, one has
$$
[0, \varphi(a_{1})]\cap [0, \varphi(a_{2})]\subset \varphi([0, a_{1}+ \epsilon]\cap [0, a_{2}+
\epsilon]). 
$$
The above theorem can be found from the papers \cite{ALFAND, AE72}. Thus one can think Archimedean ideals in affine function spaces $A(K)$ are the generalization of $M$-ideals in affine function 
spaces $A(K)$. It is unknown that if $W$ is an Archimedean ideal in $A(K)$ space what structure do the subspace $W^{\perp}$ preserve in $A(K)^*$-space.

 Given an ordered normed space $V$ with closed cone $V^+\subset V$. Then  $V$ has $\alpha$-normal property ( i.e.
  there is a constant $\alpha >0$ such that if
$ u\leq v\leq w \text{ ~~in~~} (V,V^+) $ then one has $\Vert v\Vert \leq {\alpha}\max\{\Vert u\Vert, \Vert w \Vert\}$)
if and only if $V^*$ satisfies $\alpha$-generating property ( i.e.
there is a constant $\alpha >0$ such that for each $v\in V$, there are $v_{1}, v_{2}\in V^+$ such that $v= v_{1}-v_{2}$ and $\Vert v_{1}\Vert+ \Vert v_{2}\Vert \leq \alpha \Vert v\Vert$).
The above $1$-normal property is important because it occurs in every self-adjoint part of 
$\mathrm{C}^*$-algebra \cite{KR97} and in every affine function space $A(K)$ \cite{Alf}. On the other hand, $1$-generating occurs in the dual of the self adjoint part of every $\mathrm{C}^*$-algebra and also in dual of every $A(K)$-space.

 In 2010, A. Karn \cite{Karn10} proposed the geometric axioms $(O.p.1)$ and $(O.p.2)$ for $1\leq p\leq \infty$ in general ordered normed spaces. There $1$-normal properties was 
 renamed by $(O.\infty.1)$ and $1$-generating by $(OS.1.2)$. In that paper, he extended $(O.\infty.1)$ and $(O. \infty.2)$ by $(O.p.1)$ and $(O.p.2)$ respectively. One can note that every self-adjoint part of $\mathrm{C}^*$-algebra, and $A(K)$
 has $(OS. \infty.2)$ properties.
 There is an advantage of studying geometric properties $(O.p.1)$ and $(O.p.2)$ (for  $1\leq  p < \infty$) in general 
 ordered normed spaces because the classical Banach spaces $L^{p}(\mu)$ for $1\leq p < \infty$ and self-adjoint part of trace $p$-class operators
  ( i.e. $\mathcal{T}_{p}(H)_{sa}=\{T\in \B(H)_{sa}: \tr(\vert T\vert^p)<\infty\})$ for $1\leq p <\infty$ has such geometric properties. For more literature about order smooth $p$-normed spaces
   one can see the work of A. K. Karn \cite{KARN16, KARN14}.
   
  Let $(V, V^+, e)$ be an order unit space with order unit $e$. Then $V$ is an order normed space (normed given by the order unit $e$) having $(O.\infty.1)$ and $(OS.\infty.2)$ geometric properties \cite{PT09}.
 Let $(V, \{M_{n}(V)^+\}, e)$ be an abstract operator system. Then for each $n$, $(M_{n}(V)_{sa}, M_{n}(V)^+, e^n)$ is an order unit space. Thus it has $(O.\infty.1)$ and $(OS.\infty.2)$ geometric properties \cite{PTT11}. 

In this paper, we attempt the following:
  \begin{enumerate}
\item [(i)] To provide a general framework for providing notion of ideals on order smooth $p$-normed spaces which we shall call as ``smooth $p$-order ideal''.
  \item [(ii)] To establish duality theories for such ideals.
\end{enumerate}

Note that $M$-ideals can be studied in order smooth $\infty$-normed spaces (see e.g. \cite{AA18}). But non trivial $M$-ideals may not exist in order smooth $p$-normed spaces for $1\leq p<\infty.$ 
(see e.g. \cite[Theorem $1.8$]{HWW}). The study of ideals in $A(K)$ spaces and ordered normed spaces also can be found in the works of Bonsal \cite{BNSL56, BNSL54}.
 Such notions of ideals are not possible to bring in the context of order smooth $p$-normed spaces.

Thus for first problem, we substitute the definition of the Archimedean ideals in order unit spaces in the context of order smooth $\infty$-normed spaces by ``smooth $\infty$-order ideals''.
 Also we extend the notion of smooth $\infty$-order ideals in order smooth $\infty$-normed spaces by smooth $p$-order ideals in order smooth $p$-normed spaces for $1\leq p\leq \infty$ (see e.g. Definition \ref{spoi}).
 For $p=1$, it generalize the notion of $L$-summands in order smooth $1$-normed spaces (see Theorem \ref{d11}). For $p=\infty$, it will generalize the notion of $M$-ideals in order smooth $\infty$-normed spaces  
 under certain condition (see e.g. Proposition \ref{d7} and Theorem \ref{d9}). This condition can be redundant in $A(K)$-spaces (see e.g. Remark \ref{d10}).
 
For second problem, we prove that if $W$ is a closed subspace of an order smooth $p$-normed space. Then $W$ is a smooth $p$-order ideal if and only if $W^{\perp}$ is a smooth $p'$ order 
ideal in $V^*$ (see e.g. Theorem \ref{b2}) if and only if $W^{\perp\perp}$ is an smooth $p$-order ideal in $V^{**}$(under the certain conditions)(see e.g. Theorem \ref{b5}).

\section{Preliminaries}
 A subset $C$ of a real vector space $V$ is a  \emph{cone} if  $\lambda x+y\in C$ whenever $x,y\in C$ and $\lambda \in \R^+$. A cone $C$ is \emph{proper} if $C\cap -C=\{0\}.$ An \emph{ordered vector space} is pair $(V, V^{+})$,
 where $V$ is a real vector space and $V^+$ is a cone. Let $V$ be an ordered vector space and $V^*$ be the dual of $V$. Then $V^*$ is an ordered vector space together with cone $V^+=\{f\in V^{*}: f(v)\geq 0~ \forall v \in V^{+}\}$.
 We define order relation $x\leq y$ if $y-x\in V^+$. If $W$ is subspace of ordered vector space $V$, then $W$ is also an ordered vector space together with cone $W^+= W\cap 
 V^{+}$. Let $\varphi_{W}: V \mapsto V/W$ be the canonical homomorphism. Then $V/W$ is also an ordered vector space together with cone $\varphi_{W}(V^+)$. We say that $W$ is an \emph{order ideal} in $V$ if $u,v\in W$
 and $w \in V$ such that $u\leq v\leq w$ imply $w\in W$. For basic idea regarding order structures one can see from the book \cite{GJEM}. 

A \emph{ordered normed space} is a triple $(V, V^+, \Vert  . \Vert  )$, where $(V,  \Vert  . \Vert  )$ is a normed linear space and $V^+$ is a cone. The natural way to define a cone $V^{*+}$ on $V^{*}$ by 
 $$V^{*+}:=\{f\in V^{*}: f(v)\geq 0~ \forall v \in V^{+}\}.$$
 We note that $V^{*+}$ is a $w^*$-closed set in $V^*$. In particular, if $V^{+}$ is a cone in ordered normed space. Then $V^{* +}$ is a norm closed cone in $V^{*}$.
 Let $(V, V^+,  \Vert  . \Vert  )$ be an ordered normed space. Then we can define another cone  $V_{+}$ on $V$ by 
 $$V_{+}=\{v\in V: f(v)\geq 0~ \forall f\in V^{*+}\}.$$
 Similarly,  we can define cones for $V^{*}_{+}, V^{**}_{+}$ and so on. The following proposition connect relation between cones.
\begin{proposition}\label{a1}
Let $(V,V^+,  \Vert  . \Vert  )$ be an ordered normed space. Then  we have following:
\begin{itemize}
\item [(i)] If $V^+$ is norm closed, then $V^{+}=V_{+}$; 
\item [(ii)]  $V^{*+}=V^{*}_{+}$;
\item [(iii)] $V^{**+}=V^{**}_{+}$.
\end{itemize} 
\end{proposition} 
\begin{proof}
By definition, $V^{+} \subseteq V_{+}$. If possible, let $v\in V_{+}\setminus V^{+}$.  Thus by the Hahn Banach Separation theorem, there is a $f\in V^{*}$ such that $f(v)<0$  and $f(w)\geq 0$ for all $w\in V^{+}$ so that 
$f \in V^{*+}$. Since $v\in V_{+}$ and $f\in V^{*+}$, we have $f(v)\geq 0$ which is a contradiction. Hence $V^{+}=V_{+}$. 
Same technique can be used to prove (ii) and (iii). 
\end{proof}
 Let $V$ be a Banach space and $V^*$ be its Banach dual. Let $W$ be a closed subspace of $V$. Then we have following Banach space isometry:
\begin{itemize}
 \item [(i)]$(V/W)^*\cong W^{\perp}$; 
     \item [(ii)] $W^{*}\cong V^{*}/W^{\perp}$;
     \item [(iii)]  $(V/W)^{**}\cong(W^{\perp})^{*}\cong V^{**}/W^{\perp\perp}$;
   \item [(iv)]  $W^{**}=(V^{*}/W^{\perp})^{*}\cong W^{\perp\perp}$.
\end{itemize} 
Let $V$ be an ordered normed space with closed cone $V^+$ and $W$ be a closed subspace of $V$. Then $W, W^{\perp}, W^{\perp\perp}$ are also ordered normed spaces with closed cone given by
 $W^{+}= W\cap V^+, W^{\perp+}= W^{\perp}\cap V^{*+}$ and $W^{\perp\perp+}=W^{\perp\perp}\cap V^{**+}$ respectively.
Let $\varphi_{W}:V\mapsto V/ W$, $\varphi_{W^{\perp}}:V^{*}\mapsto V^{*}/ W^{\perp}$ and $\varphi_{W^{\perp\perp}}:V^{**} \mapsto V^{**}/W^{\perp\perp}$
be the natural homomorphisms. Then $\varphi{(V^+)}$ is not a closed cone in $V/W$ and so for $\varphi_{W^{\perp}}(V^{*+}),$ and $\varphi_{W^{\perp\perp}}(V^{**+})$. The following proposition describe the topological properties
of the possible cones of the quotient spaces of the given ordered normed spaces.
\begin{proposition}\label{a3}
Let $V$ be an ordered normed space and $W$ be a closed subspace of $V$. If $\varphi_{W}, \varphi_{W^{\perp}},\varphi_{W^{\perp\perp}}$ be the natural homomorphisms. Then we have following cone relations:
\begin{itemize}
\item [(i)] $\overline{\varphi_{W}(V^{+})}^{ \Vert  . \Vert  }=\overline{\varphi_{W}(V^{+})}^{w}$;
 \item [(ii)] $\overline{\varphi_{W^{\perp}}(V^{*+})}^{ \Vert  . \Vert  }=\overline{\varphi_{W^{\perp}}(V^{*+})}^{w}$
 and
  \item [(iii)]  $\overline{\varphi_{W^{\perp\perp}}(V^{**+})}^{ \Vert  . \Vert  }=\overline{\varphi_{W^{\perp\perp}}(V^{**+})}^{w}$. 
 \end{itemize}
 \begin{proof}
It is sufficient to prove (i), as same arguments can be used for (ii) and (ii).  
 Let $v+ W\in \overline{\varphi_{W}(V^+)}^{ \Vert . \Vert }$. Then there exist $v_{n}\in V^{+}$ such that ${v_{n}+ W}$ convergent to $v+ W$ in norm.  Thus $ \Vert  f(v_{n}-v) \Vert  \rightarrow 0$ for all $f\in W^{\perp}$. 
 Thus $v_{n}+ W\rightarrow v+ W $ in $w$-topology. Thus $ \overline{\varphi_{W}(V^+)}^{ \Vert . \Vert }\subset  \overline{\varphi_{W}(V^+)}^{w}$.  
 
 Conversely, if possible let $v+ W\in \overline{\varphi_{W}(V^+)}^{w}\setminus \overline{\varphi_{W}(V^+)}^{ \Vert . \Vert }$. Then by the Hahn Banach separation theorem, there is $f\in W^{\perp}$
 such that $f(v)<0 $ and $f(u)\geq 0$ for all $u\in V^{+}$. Thus $f\in V^{*+}$.  
 Since $v+ W\in \overline{\varphi_{W}(V^{*+})}^{w}$, thus there exist a net $\{v_{\alpha}+ W\}$, where $v_{\alpha}\in V^{+}$
 such that $v_{\alpha}+W\rightarrow v+W$ in $w$-topology. Thus we have $f(v_{\alpha})\rightarrow f(v)$. Since $f(v_{\alpha})\geq 0$, we have $f(v)\geq 0$, which is a contradiction. Hence $\overline{\varphi_{W}(V^{+})}^{ \Vert  . \Vert  }=\overline{\varphi_{W}(V^{+})}^{w}$.
 \end{proof}
\end{proposition}
 
  \begin{definition}\label{a4}
  Let $V$ be an ordered normed space and $W$ be a closed subspace of $V$. Let $\varphi_{W}:V\mapsto V/ W$, $\varphi_{W^{\perp}}:V^{*}\mapsto V^{*}/ W^{\perp}$ and 
  $\varphi_{W^{\perp\perp}}:V^{**} \mapsto V^{**}/W^{\perp\perp}$ be the natural homomorphisms. Then we define order structure on $V/W, V^*/W^{\perp}$ and $V^{**}/W^{\perp\perp}$ as 
   \begin{enumerate}
   \item [(i)] $(V/W)^{+}:=\overline{\varphi_{W}(V^+)}^{ \Vert  . \Vert }$;
   \item [(ii)] $(V^*/W^{\perp})^{+}:=\overline{\varphi_{W^{\perp}}(V^{*+})}^{w^*}$;
   \item [(iii)] $(V^{**}/W^{\perp\perp})^{+}:=\overline{\varphi_{W^{\perp\perp}}(V^{**+})}^{w^*}$.
\end{enumerate} 
\end{definition}

The discussion above the Proposition \ref{a3} says about the fact that if Banach space have pre-dual, the its cone should be $w^*$-closed set. Owing to this fact in mind, we propose the definition of cones for $V^*/W^{\perp}$
 and $V^{**}/W^{\perp\perp}$ in Definition \ref{a4}. We call such cone as \emph{quotient cone} (as this cones are constructed by the help of quotient homomorphism).
 
\begin{definition}\cite{Karn10}
 Let $(V,V^+)$ be a real ordered vector space such that $V^+$ is a proper, generating and let $ \Vert  . \Vert  $ be a norm on $V$ such that $V^+$ is closed. For fixed real number $p (1\leq p< \infty)$ 
 consider the following two geometric properties on $V$:
      
         \begin{enumerate}
                 \item [(i)] $(O.p.1)$ For $u,v,w$ with $u\leq v\leq w$, we have $  \Vert  v \Vert  \leq ( \Vert   u \Vert  ^p+ \Vert   w \Vert  ^p)^{\frac{1}{p}}$;
                      \item [(ii)]$(O.p.2)$ For $v\in V$ and $\epsilon> 0$, there are $v_1,v_2\in V^+$ such that $v=v_1-v_2 $ and $( \Vert  v_1 \Vert  ^p+ \Vert  v_2 \Vert  ^p)^\frac{1}{p}< \Vert  v \Vert  + \epsilon$;
                  \item [(iii)]$(OS.p.2)$ For $v\in V$, there are $v_1,v_2\in V^+$ such that $v=v_1-v_2$ and $( \Vert  v_1 \Vert  ^p+ \Vert  v_2 \Vert  ^p)^\frac{1}{p}\leq  \Vert  v \Vert  $.
         \end{enumerate}
      For $p=\infty$, consider the similar conditions on $V$:
      
            \begin{enumerate}
                  \item [(i)]$(O.\infty.1)$ For $u,v,w$ with $u\leq v\leq w$, we have  $  \Vert  v \Vert  \leq \max \{ \Vert  u \Vert  , \Vert  w \Vert  \}$;
                      \item [(ii)]$(O.\infty.2)$ For $v\in V$ and $\epsilon> 0$, there exist $v_1, v_2\in V^+$ such that $v=v_1-v_2 $ and $\max\{ \Vert  v_1 \Vert  , \Vert  v_2 \Vert  \}<  \Vert  v \Vert  +\epsilon$;
                  \item [(iii)] $(OS.\infty.2)$ For $v\in V$, there are $v_1, v_2\in V^+$ such that $v=v_1-v_2$ and $\max( \Vert  v_1 \Vert  ,  \Vert  v_2 \Vert  ) \\
                  \leq  \Vert  v \Vert $.
             \end{enumerate}
 \end{definition}
 
 \begin{theorem}\cite{Karn10}\label{p-theory}
    Let $(V,V^+,  \Vert . \Vert )$ be a real ordered vector space such that $V^{+}$ is proper and generating. Let $ \Vert  . \Vert  $ be a norm on $V$ such that $V^+$ is closed. For fixed real number 
    $p (1\leq p\leq \infty)$, we have
          \begin{enumerate}
          \item [(i)] $ \Vert  . \Vert  $ satisfies $(O.p.1)$ condition on $V$ if and only if $ \Vert  . \Vert  ^{*}$ satisfies the condition $(OS.p'.2)$ on the Banach dual $(V^*, V^{*+}, \Vert.\Vert)$.
                 \item [(ii)] $ \Vert  . \Vert  $ satisfies the condition $(O.p.2)$ on $V$ if and only if $ \Vert  . \Vert  ^{*}$ satisfies the condition $(O.p'.1)$ on $(V^*, V^{*+}, \Vert.\Vert)$.
           \end{enumerate}
   \end{theorem} 
 
 \begin{definition}\cite{Karn10}
  Let $(V,V^+,  \Vert . \Vert )$ be a real ordered vector space such that $V^+$ is proper, generating and let $ \Vert  . \Vert  $ be a norm on $V$ such that $V^+$ is closed. For a fixed $p$, $1\leq p\leq  \infty$, 
  we say that $V$ is  an 
\emph{ordered smooth $p$-normed space}, if $ \Vert  . \Vert  $ satisfies the conditions $(O.p.1)$ and $(O.p.2)$ on $V$.
\end{definition}
Form earlier discussion in Section 1., it is clear that every self adjoint part of $\mathrm{C}^*$-algebra and affine functions space $A(K)$ and operator system are the examples of order smooth $\infty$-normed spaces.
Along with that the classical $L^{p}(\mu)$ and space of trace $p$-class operators $\T_{p}(H)$ are the example of order smooth $p$-normed spaces for $ 1\leq p <\infty.$
\begin{theorem}\label{duality-of-p-thry}\cite{Karn10} 
 Let $(V,V^+,  \Vert . \Vert )$ be a real ordered vector space such that $V^+$ is proper, generating and let $ \Vert  . \Vert  $ be a norm on $V$ such that $V^+$ is closed. For a fixed
  $p, 1\leq p\leq\infty$, $(V, V^+, \Vert.\Vert)$ is an order smooth $p$-normed space if and only if its Banach dual $(V^{*}, V^{*+}, \Vert.\Vert^*)$ is an order smooth $p'$-normed space satisfying the condition $(OS.p'.2)$. 
\end{theorem}
 In above theorem, we use $\Vert. \Vert^*$ to denote the norm of the dual Banach spaces. Rest part of the paper, we use $\Vert.\Vert$ for $\Vert.\Vert^*$.


\section{Smooth $p$-order ideals in order smooth $p$-normed spaces}
In this section, we observing the duality properties in Theorem \ref{b2} and Theorem \ref{b5}. Such duality motivate us to propose the definition of ``smooth $p$-order ideals'' in order smooth $p$-normed spaces (see e.g. Definition \ref{spoi}).
Throughout we assume that $V$ is an ordered Banach space and $W$ is a closed subspace of $V$.

The following Lemma link between quotient cone constructed from given ordered normed spaces with the dual cone constructed from duality ordered normed spaces.
  \begin{lemma}\label{b1}Let $(V,V^+,  \Vert . \Vert )$ be an order smooth $p$-normed
 space and $W$ be a subspace of $V$. Let $\varphi_{W}: V\mapsto V/W$ and $\varphi_{W^{\perp}}: V^{*} \mapsto V^{*}/W^{\perp}$ be the natural
 homomorphisms. Then we have following:
\begin{enumerate}
\item [(i)] $\{ f+W^{\perp}: f(w)\geq 0~ \forall w\in W^{+}\}=(V^*/W^{\perp})^{+}$;
\item [(ii)] $\{f \in W^{\perp}: f(v) \geq 0 ~\forall v+W\in (V/W)^+ \}=W^{\perp+}$.
\end{enumerate}
\end{lemma}

\begin{proof}
Let $(V,V^+)$ be an order smooth $p$-normed space and $W$ be a subspace of $V$.
\begin{enumerate}
\item [(i)] We note that Banach dual of $W$ is $V^{*}/W^{\perp}$. We claim that $\{f+ W^{\perp}: f(w)\geq 0 ~\forall w\in W^{+} \}$ is a $w^*$-closed set. Let $\{f_{\alpha}+W^{\perp}\}$ be a net in $\{f+ W^{\perp}: 
f(w)\geq 0 ~\forall w\in W^{+} \}$ such that $f_{\alpha}+ W^{\perp}\rightarrow f+ W^{\perp}$ for some $f\in V^{*}$ in $w^*$-topology. Since $W$ is a predual of $V^{*}/W^{\perp}$, thus $f_{\alpha}(w)\rightarrow f(w)$ for all
$w\in W^{+}$. Since $f_{\alpha}(w)\geq 0$ for all $w\in W^{+}$, therefore $f(w)\geq 0$ for all $w\in W^{+}$. Hence $\{f+ W^{\perp}: f(w)\geq 0 ~\forall w\in W^{+} \}$ is a $w^*$-closed set.
We know from definition that $(V^*/W^{\perp})^{+}=\overline{\varphi_{W^{\perp}}(V^{*+})}^{w^*}$. Let $f\in V^{*+}$. Since $f(w)\geq 0$ for all $w\in W^{+}$. Thus $f+ W^{\perp}\in \{f+ W^{\perp}: 
f(w)\geq 0 ~\forall w\in W^{+} \}.$ Thus we have $\overline{\varphi_{W^{\perp}}(V^{*+})}^{w^*}\subset  \{f+ W^{\perp}: f(w)\geq 0 ~\forall w\in W^{+} \}$.   

Conversely, if possible let $f+ W^{\perp}\in \{f+ W^{\perp}: f(w)\geq 0 ~\forall w\in W^{+} \}\setminus\overline{\varphi_{W^{\perp}}(V^{*+})}^{w^*}$. Then by the Hahn Banach separation theorem, there is a $w\in W$ 
such that $f(w)<0$ and $g(w)\geq 0$ for all $g\in V^{*+}$.  Therefore $w\in V^{+}$ by the Proposition \ref{a1}. Thus $v\in W\cap V^+ =W^{+}$. Therefore $f(w)\geq 0$, which is a contradiction. Hence $\{ f+W^{\perp}: f(w)\geq 0~ \forall w\in W^{+}\}=
(V^*/W^{\perp})^{+} $.
\item [(ii)]
Let $f\in W^{\perp+}$. Then $f\in W^{\perp}$ and $f\in V^{*+}$. Thus $f(v)\geq 0$ for all $v\in V^+$ so that $f(v)\geq 0$ for all $v+ W\in \varphi_{W}(V^+)$. We claim that $f(v)\geq 0$ for all $v+ W\in 
\overline{\varphi_{W}(V^+)}^{ \Vert . \Vert }$. Let $v+ W\in \overline{\varphi_{W}(V^+)}^{ \Vert . \Vert }$. Then there is a sequence $v_{n}+W \in \varphi_{W}(V^+)$ such that $v_{n}+ W\rightarrow v+W$ in norm.
Since $f\in W^{\perp}$, we have $f(v_{n})\longrightarrow f(v)$. Since $f(v_{n})\geq 0$ for all $n\in \N$, therefore we have $f(v) \geq 0$. We know from the definition that $(V/W)^+=\overline{\varphi_{W}(V^+)}^{ \Vert . \Vert }$. Thus
$W^{\perp+}\subset \{f\in W^{\perp}:f(v)\geq 0,~ \forall v+W\in (V/W)^{+}\}$.

Conversely, let $f\in \{f\in W^{\perp}:f(v)\geq 0,~ \forall v+W\in (V/W)^{+}\}$. Now if $v\in V^{+}$, then $v+ W\in (V/W)^+$. This implies that $f(v)\geq 0$ for all $v\in V^+$ so that $f\in V^{*+}$. Therefore
$f\in V^{*+}\cap W^{\perp}=W^{\perp+}.$
 \end{enumerate}
\end{proof}
\begin{mysection}
\begin{align*}
\{f\in W^{\perp}:f(v)\geq 0,~ \forall v+W\in (V/W)^{+}\} &=\{f\in W^{\perp}:f(v)\geq 0~\forall v+W\in \overline{\phi_{W}(V^+)}^{ \Vert  . \Vert  }\}\\
&=\{f\in W^{\perp}:f(v)\geq 0~\forall v+W\in \phi_{W}(V^+)\}\\
&=\{f\in W^{\perp}:f(v)\geq 0~ \forall v\in V^{+}\}\\
&=W^{\perp}\cap V^{'+}=W^{\perp+}.
\end{align*}
\end{mysection}
 \begin{theorem}\label{b2}
Let $(V,V^+,  \Vert . \Vert )$ be an order smooth $p$-normed space,  $W$ be a subspace of $V$. Let $\varphi_{W}: V\mapsto V/W$ and $\varphi_{W^{\perp}}: V^{*} \mapsto V^{*}/W^{\perp}$ be the natural
 homomorphisms. Then we have following duality:
\begin{enumerate}
\item [(i)] $(W, W^+, \Vert  . \Vert  )$ is an order smooth $p$-normed space if and only if $$(V^{*}/W^{\perp},(V^{*}/W^{\perp})^+, \Vert  . \Vert  )$$ is an order smooth $p'$-normed space satisfying $(OS.p'.2)$.
\item  [(ii)]$(V/W, (V/W)^+, \Vert  . \Vert  )$ is an order smooth $p$-normed space if and only if $(W^{\perp}, W^{\perp+},  \\ \Vert  . \Vert )$ is an order smooth $p'$-normed space satisfying $(OS.p'.2)$.
\end{enumerate}
\end{theorem}
\begin{proof} Let $W$ be a subspace of an ordered smooth $p$-normed space $(V, V^{+},  \Vert  . \Vert  )$.
\begin{enumerate}
\item [(i)]Since Banach dual of $W$ is $V^*/W^{\perp}$ and from (i) of Lemma \ref{b1}, we know that $\{ f+W^{\perp}: f(w)\geq 0~ \forall w\in W^{+}\}=(V^*/W^{\perp})^{+}$. Thus by applying Theorem \ref{p-theory} between
 $(W, W^{+},\Vert.\Vert)$ and $( V^{*}/W^{\perp}),V^{*}/W^{\perp})^+, \Vert  . \Vert )$, we conclude that 
 $(W, W^+, \Vert  . \Vert  )$ is an order smooth $p$-normed space if and only if $(V^{*}/W^{\perp},(V^{*}/W^{\perp})^+, \\ \Vert  . \Vert  )$ is an order smooth $p'$-normed space satisfying $(OS.p'.2)$.
\item [(ii)] Since Banach dual of $V/W$ is $W^{\perp}$ and  from Lemma \ref{b1}, we know that  $\{f \in W^{\perp}: f(v) \geq 0 ~\forall v+W\in (V/W)^+ \}=W^{\perp+}$. Thus by applying Theorem \ref{p-theory}
between $(V/W, (V/W)^+, \Vert  . \Vert )$ and $(W^{\perp}, W^{\perp+},  \Vert . \Vert )$, we conclude that $(V/W,(V/W)^+,  \Vert  . \Vert  )$ is an order smooth $p$-normed space if and only if $(W^{\perp}, W^{\perp+},  \Vert . \Vert )$ is 
an order smooth $p'$-normed space satisfying $(OS.p'.2)$.
\end{enumerate}
\end{proof}
\begin{proposition}\label{b3}Let $(V,V^+, \Vert . \Vert )$ be an order smooth $p$-normed space and $W$ be a subspace of $V$. Then $\varphi_{W^{\perp}}(V^{*+})= \overline{\varphi_{W^{\perp}}(V^{*+})}^{w^*}$ if and only if
 $f\in W^{*+}$ implies there is a $g\in V^{*+}$ such that $g_{|_W}=f$.
\end{proposition}
\begin{proof} Let $\varphi_{W^{\perp}}(V^{*+})= \overline{\varphi_{W^{\perp}}(V^{*+})}^{w^*}$. Let $f: W\mapsto \R$ be a bounded linear functional such that $f(w)\geq 0$ for all $w\in W^{+}$. Then by the Hahn Banach 
separation theorem, there exist a bounded linear functional $f_{1}: V\mapsto \R$ such that ${f_{1}}_{|_W}= f$ and $ \Vert  f_{1} \Vert =  \Vert  f \Vert $. Now by Lemma \ref{b1}, $f_{1}+ W^{\perp}\in \overline{\varphi_{W^{\perp}}(V^{*+})}
^{w^*}$. Thus by assumption, there is a $g\in V^{*+}$ such that $f_{1}+ W^{\perp}= g+ W^{\perp}$. Therefore we have ${f_{1}}_{|_W}= g_{|_{W}}$. 

Conversely, assume that  if $f\in W^{*+}$, then there is a $g\in V^{*+}$ such that $g_{|_W}=f$.  Now let $f+ W^{\perp}\in \overline{\varphi_{W^{\perp}}(V^{*+})}^{w^*}$. Then there exist $g_{\alpha}\in V^{*+}$ such that
$g_{\alpha}+ W^{\perp}\longrightarrow f+ W^{\perp}$ in $w^{*}$-topology which implies that $f(w)\geq 0$ for all $w\in W^{+}$.  Thus $g_{\alpha}(w)\longrightarrow f(w)$ for all $w\in W^{+}$. So by assumption, there exist a $g\in V^{*+}$ such that $g_{|W}= f_{|W}$.
Therefore $g+ W= f+ W$ so that $f+ W^{\perp}\in \varphi_{W^{\perp}}(V^{*+})$.  
\end{proof}
\begin{lemma}\label{b4} Let $(V,V^+, \Vert . \Vert )$ be an order smooth $p$-normed space and $W$ be a subspace of $V$. Then we have following:

  $$\{F+{W^{\perp\perp}: F(f) \geq 0 ~\forall f\in W^{\perp+}}\}=(V^{**}/W^{\perp\perp})^+.$$

 \end{lemma}
 \begin{proof}Let $(V,V^+, \Vert . \Vert )$ be an order smooth $p$-normed space and $W$ be a subspace of $V$.
We know from definition \ref{a4}, that $(V^{**}/W^{\perp\perp})^+=\overline{\varphi_{W^{\perp\perp}}(V^{**+})}^{w^*}$. 
 It is clear from the definition that $\varphi(V^{**+})\subset \{F+{W^{\perp\perp}: F(f) \geq 0 ~\forall f\in W^{\perp+}}\}.$ Since $\{F+{W^{\perp\perp}: F(f) \geq 0 ~ \\ \forall  f\in  W^{\perp+}}\}$ is a $w^*$-closed set, therefore
 we have $\overline{\varphi(V^{**+})}^{w^*}  \subset \{F+{W^{\perp\perp}:F(f) \geq 0 ~\forall f\in W^{\perp+}}\}$.

If possible, let $G+W^{\perp\perp}\in\{F+{W^{\perp\perp}: F(f) \\ \geq 0 ~\forall f\in W^{\perp+}}\}\setminus \overline{\varphi_{W^{\perp\perp}}(V^{**+})}^{w^*}$. Since $W^{\perp}$ is a predual of
${V^{**}/W^{\perp\perp}}$. Thus by the Hahn Banach separation theorem, there exist $g\in W^{\perp}$ such that 
$G(g)< 0$ and $F(g)\geq 0$ for all $F+W^{\perp\perp}\in \overline{\varphi_{W^{\perp\perp}}(V^{**+})}^{w^{**}}.$
Thus $F(g)\geq 0$ for all $F\in V^{**+}$. Therefore from Proposition \ref{a1}, we have $g\in W^{\perp+}$ so that $G(g)\geq 0$, which is a contradiction. Hence we have $\{F+{W^{\perp\perp}: F(f) \geq 0 ~\forall f\in W^{\perp+}}\}=(V^{**}/W^{\perp\perp})^+$.

 \end{proof}
\begin{theorem}\label{b5}
Let $(V,V^+, \Vert . \Vert )$ be an order smooth $p$-normed space and $W$ be a subspace of $V$. Let $\varphi_{W}: V\mapsto V/W$ and $\varphi_{W^{\perp}}: V^{*} \mapsto V^{*}/W^{\perp}$ be the natural homomorphisms. 
Then we have following duality:
\begin{enumerate}
\item [(i)] $(W^{\perp}, W^{\perp+}, \Vert  . \Vert  )$ is an order smooth $p'$-normed space if and only if $$(V^{**}/W^{\perp\perp},(V^{**}/W^{\perp\perp})^+, \Vert  . \Vert )$$ is an order smooth $p$-normed 
space satisfying $(OS.p.2)$;
\item [(ii)] if $(V^{*}/W^{\perp}, (V^{*}/W^{\perp})^+, \Vert  . \Vert  )$ is an order smooth $p'$-normed space, then $$(W^{\perp\perp}, W^{\perp\perp+},  \Vert  . \Vert )$$ is an order smooth $p$-normed space satisfying $(OS.p.2)$;
\item [(iii)] assume that $\varphi_{W^{\perp}}(V^{*+})= \overline{\varphi_{W^{\perp}}(V^{*+})}^{w^*}$. If $(W^{\perp\perp},W^{\perp\perp+},  \Vert . \Vert )$ is an order smooth $p$-normed space, then 
$(V^*/W^{\perp}, (V^{*}/W^{\perp})^+,  \Vert . \Vert )$ is an order smooth $p'$-normed space.  
\end{enumerate}
\end{theorem}

\begin{proof}
\begin{enumerate}
\item [(i)] Since Banach dual of $W^{\perp}$ is ${V^{**}/W^{\perp\perp}}$. We know from (i) of Lemma \ref{b4} that $\{F+{W^{\perp\perp}: F(f) \geq 0 ~\forall f\in W^{\perp+}}\}=(V^{**}/W^{\perp\perp})^+$.
Therefore from Theorem \ref{p-theory}, we conclude that $(W^{\perp}, W^{\perp+},  \Vert . \Vert )$ is an order smooth $p'$-normed space if and only if $(V^{**}/W^{\perp\perp}, (V^{**}/W^{\perp\perp})^{+},  \Vert . \Vert )$ is
an order smooth $p$-normed space satisfying $(OS.p.2)$. 
\item [(ii)] We have following cone relation on $W^{\perp\perp}$.
\begin{align}
\{F\in W^{\perp\perp}: F(f)\geq 0,~ \forall & f+W\in (V^{*}/W^{\perp})^{+}\}\\   \notag
&=\{F\in W^{\perp\perp}:F(f)\geq 0~\forall f+W^{\perp}\in \overline{\varphi_{W^{\perp}}(V^{*+})}^{w^*}\}\\ \notag
&\subset\{F\in W^{\perp\perp}:F(f)\geq 0~\forall f+W^{\perp}\in \varphi_{W^{\perp}}(V^{*+})\}\\ \notag
&= \{F\in W^{\perp\perp}:F(f)\geq 0~ \forall f\in V^{*+}\}\\ \notag
&=W^{\perp\perp}\cap V^{**+}=W^{\perp\perp+}.
\end{align} 
We know that $W^{\perp\perp+}$ is proper and closed. Since $(V^{*}/W, (V^{*}/W)^+,  \Vert . \Vert )$ is an order smooth $p'$-normed space, by Theorem \ref{p-theory}, $(W^{\perp\perp},  \Vert . \Vert )$ is 
an order smooth $p'$-normed space satisfying $(OS.p.2)$ with respect to cone $\{F\in W^{\perp\perp}: F(f)\geq 0,~ \forall f+W\in (V^{*}/W^{\perp})^{+}\}$. Since $ \{F\in W^{\perp\perp}: 
F(f)\geq 0,~ \forall f+W\in (V^{*}/W^{\perp})^{+}\} \subset W^{\perp\perp+}$ and $W^{\perp\perp+}$ is a proper closed cone, thus $(W^{\perp\perp}, W^{\perp\perp+}, \Vert . \Vert )$ is also an order smooth
 $p$-normed space satisfying $(OS.p.2)$.   
 
\item [(iii)] Let $\varphi_{W^{\perp}}(V^{*+})= \overline{\varphi_{W^{\perp}}(V^{*+})}^{w^*}$. Then from equation (i) of item $(b)$, we can easily check that $\{F\in W^{\perp\perp}: F(f)\geq 0,~ \forall f+W\in (V^{*}/W^{\perp})^{+}\}=W^{\perp\perp+}.$
Therefore if $(W^{\perp\perp}, W^{\perp\perp+}, \Vert.\Vert)$ is an order smooth $p$-normed space, then its predual $(V^*/W^{\perp}, (V^{*}/W^{\perp})^+,  \Vert . \Vert )$ is an order smooth $p'$-normed space.
\end{enumerate} 
\end{proof}

\begin{corollary}
Let $(V,V^+,  \Vert . \Vert )$ be an order smooth $p$-normed space and $W$ be an subspace of $V$. Then following are equivalents:
\begin{enumerate}

\item [(i)] Then $(W, W^{+},  \Vert . \Vert )$ is an order smooth $p$-normed space if and only if $(W^{\perp\perp}, W^{\perp\perp+}, \\ \Vert . \Vert  )$ is an order smooth $p$-normed space of $V^{**}$ satisfying $(OS.p.2)$.

\item [(ii)] If $\varphi_{W^{\perp}}(V^{*+})= \overline{\varphi_{W^{\perp}}(V^{*+})}^{w^*}$, then $(V/W, (V/W)^+,  \Vert  . \Vert  )$ is an order smooth $p$-normed space if and only if $(V^{**}/W^{\perp\perp}, ( V^{**}/W^{\perp\perp})^+,  \Vert  . \Vert  )$ is an order smooth $p$-normed space satisfying $(OS.1.2)$.
\end{enumerate}
\end{corollary}

\begin{definition}\label{spoi}
If $(V,V^+,  \Vert  . \Vert  )$ is an order smooth $p$-normed space. Then a subspace $W$ is called \emph{smooth $p$-order ideal} in $V$ if $W$ satisfies the following conditions:
\begin{enumerate}
\item [(i)] $\varphi_{W^{\perp}}(V^{*+})= \overline{\varphi_{W^{\perp}}(V^{*+})}^{w^*};$
\item [(ii)] $(W, W^+, \Vert . \Vert )$ is an order smooth $p$-normed space;
\item [(iii)] $(V/W, (V/W)^+,  \Vert  . \Vert  )$ is an order smooth $p$-normed space.
\end{enumerate}
\end{definition}
\begin{remark} It is immediate from the definition \ref{spoi} that If $W$ is a smooth $p$-order ideal, then $W, W^{\perp}, W^{\perp\perp}$ are order ideals. 
\end{remark}

\subsection{ Smooth $\infty$-order ideals}

In this subsection, we see that the definition of smooth $\infty$-order ideals can be redundant (see e.g.  Corollary \ref{c2}) 
\begin{theorem}\cite[Theorem 2.5]{AA18}\label{c1} Let $(V,V^+, \Vert . \Vert )$ be an order smooth $\infty$-normed space and $W$ be an order smooth $\infty$-normed space. If 
 $f\in W^{*+}$, then there is a $g\in V^{*+}$ such that $g_{|_W}=f$.
\end{theorem}
 \begin{corollary}\label{c2}
      Let $(V, V^+,  \Vert . \Vert )$ be an order smooth  $\infty$-normed space and $W$ be a subspace of $V$. Then $W$ is an smooth $\infty$-order ideal if and only if $W$ satisfies following conditions:
      \begin{enumerate}
      \item [(i)]$(W, W^+, \Vert . \Vert )$ is an order smooth $\infty$-normed space;
      \item [(ii)]$(V/W, (V/W)^+,  \Vert  . \Vert  )$ is an order smooth $\infty$-normed space.
      \end{enumerate}
       \end{corollary}
       The condition (ii) of Corollary \ref{c2} has  following equivalent relations:
       
\begin{theorem}\label{c3}
 Let $(V, V^+,  \Vert . \Vert )$ be an order smooth  $\infty$-normed space and $W$ be a subspace of $V$. Then following are equivalent: 
\begin{enumerate}
 \item [(i)] $((V/W), (V/W)^+,  \Vert  . \Vert )$ is an order smooth $\infty$-normed space;
   \item [(ii)] $(W^{\perp}, W^{\perp+},  \Vert  . \Vert  )$ \textrm{~satisfying~}  $(OS.1.2)$;
      \item [(iii)] $ \Vert  v+W \Vert  =\sup \{|f(v)|:f\in (W^{\perp})_{1}\cap W^{\perp+}\}$;
          \item [(iv)] $ \Vert  F+W^{\perp\perp} \Vert  =\sup \{|F(f)|:f\in (W^{\perp})_{1}\cap W^{\perp+}\}$;
              \item [(v)] $((V^{**}/W^{\perp\perp}), (V^{**}/W^{\perp\perp})^+,  \Vert  . \Vert  )$ is an order smooth $\infty$-normed space.
       \end{enumerate}
     \end{theorem}
\begin{proof} It is clear that (i), (ii) and (iii) are equivalent and (iv) implies (iii). Thus it is sufficient to prove that $(ii)\implies (iii)$ and  $(iii)\implies (ii)$ and $(ii)\implies (iv)$. 

\item [(ii)$\implies$ (iii)]  Let $v\in V$, then we have
 \begin{align*}
 \Vert  v + W \Vert &= \sup\{|f(v)| : f\in (W^\perp)_{1}\}\\
&= \sup\{|f(v)| : f \in \co((W^{\perp+}\cap \ (W^\perp)_{1}\cup-(W^{\perp+}\cap \ (W^\perp)_{1})\}[W^{\perp} \textrm{~has~} (OS.1.2) ]\\
&= \sup\{|f(v)| : f \in (W^\perp)_{1} \cap W^{\perp+}\}.
 \end{align*}
 
\item [(iii)$\implies$ (ii)] Let $u_{1}+ W\leq u_{2}+ W$ in $((V/W), (V/W)^+,  \Vert  . \Vert  )$. By definition of $(V/W)^+$, there exist a sequence 
$\{v_{n}\}$ in $V^+$ such that $v_{n}+ W \longrightarrow (u_{2}- u_{1})+ W$ in norm topology. Now let $f\in W^{\perp+}$, then $f(v_{n}) \longrightarrow f(u_{2}- u_{1}).$ Since $f(v_{n}) \geq 0$ for each $n$, 
therefore we have $f(u_{1}) \leq f(u_{2})$.

Now, let $v_{1}+ W \leq v_{2}+ w\leq v_{3} +W$ in $((V/W), (V/W)^+,  \Vert  . \Vert )$. Then  for any $f\in W^{\perp+} \cap (W^{\perp})_{1}$, we have
$f(v_{1}) \leq f(v_{2}) \leq f(v_{3})$. Hence by assumption, we have
$\Vert v_{2} + W\Vert \leq \max\{ \Vert v_{1} + W\Vert , \Vert v_{3} + W\Vert\}$. Since $((V/W), (V/W)^+,  \Vert  .\Vert )$ has $(O.\infty.1)$ property,  thus $(W^{\perp} ,W^{\perp+}, \Vert.\Vert)$ has $(OS.1.2)$ property.
  
\item [(ii)$\implies$ (iv)]  Let $F\in V^{**}$, then we have
 \begin{align*}
 \Vert  F + W \Vert &= \sup\{|F(f)| : f\in (W^\perp)_{1}\}\\
&= \sup\{|F(f)| : f \in \co((W^{\perp+}\cap \ (W^\perp))_{1}\cup-(W^{\perp+}\cap \ (W^\perp)_{1})\}[W^{\perp} \textrm{~has~} (OS.1.2) ]\\
&= \sup\{|F(f)| : f \in (W^\perp)_{1} \cap W^{\perp+}\}.
 \end{align*}

\end{proof}

\section{$M$-ideals and  smooth $\infty$-order ideals}
Let $W$ be a closed subspace of a Banach space $V$. We call $W$  an \emph{$L$-summand} if there exists a (unique) subspace $W^{'}$ of $V$ such that 
$V=W\oplus_{1}W^{'}.$
We call $W$ an \emph{$M$-ideal} in $V$ if  $W^{\perp}$ is an $L$-summand in $V^*$. In other word,
$V^{*}= W^{\perp} \oplus_{1} W^{\perp'}$.
\begin{theorem}\label{d11}
Let $(V, V^+, \Vert.\Vert)$ be an order smooth $1$-normed space, satisfying $(OS.1.2)$ and $W$ be a subspace of  $V$. If $W$ is an $L$-summand, then $W$ is an smooth $1$-order ideal in $V$. 
\end{theorem}

 \begin{proposition}\label{d7} Let $(V, V^+,  \Vert . \Vert )$ be an order smooth $\infty$-normed space and $W$ be a subspace of $V$. If $W$ is an $M$-ideal, then $((V/W), (V/W)^+,  \Vert  . \Vert  )$ is an order smooth 
      $\infty$-normed space.
     \end{proposition}
     To prove the above proposition and theorem, we need following informations and lemmas. 
We recall few facts from \cite{AE72} which will be needed in the next lemma. Let $K$ be a convex subset of a  vector space $V$. Then a non-empty convex subset $F\subset K$ is called a \emph{face} if for any $u,v \in K,$ we have 
$u,v \in F$ whenever  $\lambda u+(1-\lambda)v\in F$ for some $ 0 < \lambda < 1.$ We define $$\cone (K) :=\cup_{\lambda\geq 0}\lambda K$$ is the smallest cone containing $K.$ 
 If $v\in V$ and $v\neq 0,$ we define
$$ \face_{V_1} (\frac{v}{\Vert v \Vert}) =\{ w\in V_{1} : \frac{v}{\Vert v \Vert}= \lambda w +(1-\lambda)u \text{ for some } \lambda \in (0,1) \text{ and some } u\in V_1\}.$$
We write $C(v) := 
\cone (\face_{V_1}(\frac{v}{\Vert v \Vert}))$ for the smallest facial cone containing $v$ if $v \not= 0$. We define $C(0) = \{ 0 \}$.  
For a cone $C$ in $V$, we write 
$$C'=\{v\in V: C\cap C(v)=\{0\}\}.$$ 
It may be noted that $C'$ may not be convex, in general. 

\begin{lemma}\label{pro-of-facl-cne}\cite[Lemma 2.3]{AE72}
	Let $V$ be a normed linear space and let $u_1,\cdots, u_n\in V.$ Then the following facts are equivalent:
	\begin{enumerate}
		\item [(i)]$u_1,\cdots, u_n\in C(u_1+\cdots+u_n).$
		\item [(ii)] $\|\Sigma_{1=1}^{n}u_i \|=\Sigma_{i=1}^{n} \|u_i \|.$
	\end{enumerate}
\end{lemma}

\begin{theorem}\cite[Part I, Theorem 2.9]{AE72}\label{cne-dec}
	Let $C$ be a norm closed convex cone in a Banach space $V.$ Then every $u\in V$ admits a decomposition $u=v+w, ~\| u \|= \|v \|+ \| v \|,$ where $v\in C$ and $w\in C'.$  
\end{theorem}

\begin{lemma}\label{c4}
  Let $(V, V^+,  \Vert . \Vert )$ be an order smooth $1$-normed space satisfying $(OS.1.2)$ and  $W$ be an $L$-summand in $V$. If  $u\in W^{}$, then  $C(u)\subseteq W^{}$.
\end{lemma}

  \begin{proof}
 Let  $w\in W\setminus \{0\}$. Let $u\in C(w)$ and without loss of generality we may assume that $ \Vert  u \Vert =1$. Then by definition of
$C(w)$, we have $u\in \face_{K}(\frac{w}{ \Vert  w \Vert })$. Thus 
there is a $v\in V_{1}$ such that $\lambda u+(1-\lambda)v=\frac{w}{ \Vert  w \Vert }$ for some $\lambda\in (0,1)$. By triangle inequality, $ \Vert  v \Vert =1= \Vert u\Vert$.
Since $u,v\in V$ and $W$ is subspace of a complete normed space, there are $u_1,v_1\in W$ and $u_2, v_2\in W^{'}$ such that
\begin{align*}
u&=u_1+u_2           &         \Vert  u \Vert  &= \Vert  u_1 \Vert + \Vert  u_2 \Vert , \\
v&=v_1+v_2           &         \Vert  v \Vert  &= \Vert  v_1 \Vert  + \Vert  v_2 \Vert .         
\end{align*}
 Now,  $\frac{w}{ \Vert  w \Vert  }=\lambda u_1+(1-\lambda)v_1 +\lambda u_2+(1-\lambda)v_2$. From it, we can rewrite as 
   
    \begin{align*}
   (\lambda u_1+(1-\lambda)v_1-\frac{w}{ \Vert  w\Vert  })+\lambda u_2=-(1-\lambda)v_2, \\
    (\lambda u_1+(1-\lambda)v_1-\frac{w}{ \Vert  w \Vert  })+(1-\lambda) v_2=-\lambda u_2. 
   \end{align*} 
   Since  $\lambda u_1+(1-\lambda)v_1-\frac{w}{  \Vert  w  \Vert }\in W \mbox{ and } u_2,v_2 \in W^{'}$ and $W$ is an $L$-summand, from last two equations, we get following norm equalities:
\begin{align*}
  \Vert  \lambda u_1+(1-\lambda)v_1-\frac{w}{ \Vert  w \Vert  } \Vert +\lambda  \Vert  u_2 \Vert  =(1-\lambda) \Vert  v_2 \Vert , \\
   \Vert  \lambda u_1+(1-\lambda)v_1-\frac{w}{ \Vert  w \Vert } \Vert +(1-\lambda)  \Vert  v_2 \Vert =\lambda \Vert  v_2 \Vert ,
  \end{align*}
  which implies that  $\lambda u_1+(1-\lambda)v_1=\frac{w}{  \Vert  w  \Vert }$. Since $ \Vert  u_1 \Vert  , \Vert  v_1 \Vert  \leq1 $, thus by triangle inequality, $ \Vert  u_1 \Vert =1= \Vert  u_2 \Vert $, so that $u_2=0=v_2$.
   Hence $C(w)\subset W$.
    \end{proof}
     \begin{lemma}\cite{AAp1}\label{c5} Let $(V, V^{+},  \Vert . \Vert )$ be an order smooth $1$-noremd space satisfying $(OS.1.2)$. If $L$ is an $L$-projection of $V^{*}$, then $L$ is a positive linear map.
    \end{lemma}
    \begin{proof} Let $u\in V^+$. Since $P$ is an $L$-projection, therefore $\Vert u\Vert = \Vert L(u)\Vert+ \Vert u- L(u)\Vert$. Thus by Lemma \ref{pro-of-facl-cne}, we have $L(u), u-L(u)\in C(u)$.
Since $u\in V^{+}$, by Lemma 2.2 of \cite{AA18}, we have $C(u)\subset V^{+}$. Therefore, we have $L(u) \geq 0$.  
\end{proof}
     
    \begin{lemma}\label{c6} Let $(V, V^+,  \Vert . \Vert )$ be an order smooth $\infty$-normed space and $W$ be an $M$-ideal in $V$ so that $V^{*}= W^{\perp}\oplus_{1}W^{\perp'}$, where $W^{\perp'}$ is the complemented 
    subspace of $W^{\perp}$. If $\varphi_{W^{\perp}}(V^{*+})= \overline{\varphi_{W^{\perp}}(V^{*+})}^{w^*}$, then $(V^*/W^{\perp},(V^{^*}/W^{\perp})^+,  \Vert . \Vert )$ is isometrically order isomorphic to 
    $(W^{\perp'}, W^{\perp'+},  \Vert . \Vert )$. 
    \end{lemma}
    \begin{proof} Let $P$ be the $L$-projection of $V^*$ onto $W^{\perp'}$. We define a map $\varphi: V^{*}/W^{\perp} \mapsto W^{\perp'}$ by
     $$\varphi(f+ W^{\perp})=P(f)$$
      for all $f\in V^{*}$. Let $f\in V^{*}$ such that 
    $P(f)= 0$. Then $f\in W^{\perp}$ so that $f+W^{\perp}= 0+ W^{\perp}$. Hence $\varphi$ is well defined. Let $f\in V^{*}$. We claim that $ \Vert  f+ W^{\perp} \Vert  =  \Vert  P(f) \Vert $. Since $f-P(f)\in W^{\perp}$,
    we have $f+W^{\perp}= P(f)+ W^{\perp}$. Since $P$ is an $L$-projection on $W^{\perp'}$, thus $ \Vert  P(f)+ g \Vert =  \Vert  P(f) \Vert +  \Vert  g \Vert $ for all $g\in W^{\perp}$. Therefore
     $ \Vert  P(f) \Vert =  \Vert  f+ W^{\perp} \Vert $ and
     $\varphi$ is isometry onto $W^{\perp'}$.  By assumption, $(V^{*}/W^{\perp})^{+}= \varphi_{W^{\perp}}(V^{*+})$. Let $f\in V^{*+}.$ Since $P$ is an $L$-projection, by Lemma \ref{c5}, $P$ is a positive map so that
    $P(f)\in W^{\perp'+}$. Hence $\varphi$ is a positive map.  
    Conversely, if $g\in W^{\perp'+}$, then $g\in V^{*+}$ and $\varphi(g+ W^{\perp})= P(g)= g$. Thus  $\varphi^{-1}$ is also positive map. Hence $\varphi: V^{^*}/W^{\perp}\mapsto W^{\perp'}$ is an isometrical order isomorphism.
    \end{proof}
    \begin{lemma}\label{c7}Let $(V, V^+,  \Vert . \Vert )$ be an order smooth $\infty$-normed space and $W$ be an $M$-ideal in $V$ so that $V^{*}= W^{\perp}\oplus_{1}W^{\perp'}$, where $W^{\perp'}$ is the complemented 
    subspace of $W^{\perp}$. Then $(W^{\perp'}, W^{\perp'+},  \Vert . \Vert )$ is an order smooth $1$-normed space.
    \end{lemma}
    \begin{proof} Since $W^{\perp'}\subset V^{*}$, $W^{\perp'}$ satisfies $(O.1.1)$. We claim that $W^{\perp'}$ satisfies $(OS.1.2)$. Let $f\in W^{\perp'}$. Since $W^{\perp'}$ is an $L$-summand of an order smooth $1$-normed 
    space $V^*$ satisfying $(OS.1.2)$, by  
    Lemma \ref{c4}, we may conclude that
    $C(f)\subset W^{\perp'}$. Since $f\in V^{*}$ and $V^*$ satisfies $(OS.1.2)$, there are $g, h\in V^{^{*+}}$ such that $f= g-h$ and  $ \Vert  f  \Vert = 
    \Vert  g \Vert +  \Vert  h  \Vert $. By  \cite[Lemma 2.3, part I]{AE72}, we have $g, -h\in C(f)$ so that $g, h\in W^{\perp'+}$. Thus $W^{\perp'}$  is an order smooth $1$-normed space
   satisfies $(OS.1.2)$.
    \end{proof}
    \begin{proof}[{\bf Proof of theorem \ref{d11}}]  To prove $W$ is an order smooth $1$-normed space, it is suffices to show that $W$ satisfies $(OS.1.2)$. Let $w\in W$. Since $V$ satisfies $(OS.1.2)$, there are $u, v\in V^+$
 such that $w=u-v$ and $\Vert w\Vert = \Vert u\Vert+ \Vert v\Vert$. By Lemma \ref{pro-of-facl-cne}, we have $u, -v\in C(w)$. Also from Lemma 
 \ref{c4}, we have $v, -w\in W$. Thus $W$ satisfying $(OS.1.2)$.

 We claim that $\varphi_{W^{\perp}}(V^{*+})= \overline{\varphi_{W^{\perp}}(V^{*+})}^{w^*}$. Let $f: W\mapsto \R$ be a bounded positive linear functional. Let $P$ be the $L$-projection of $V$ onto $W$. Then by Lemma \ref{c5}, $P$ is a positive linear map.
 Let $g(v)= f(P(v))$ for all $v\in V$. Then $g:V\mapsto \R$ is a positive linear map such that for all $w\in W$, we have $g(w)= f(P(w))= f(w)$.
 Hence by Proposition \ref{b3}, we have $\varphi_{W^{\perp}}(V^{*+})= \overline{\varphi_{W^{\perp}}(V^{*+})}^{w^*}$.
 
 Since $W$ is an $L$-summand, there is a unique subspace $W^{'}$ of $V$ such that 
 $V= W\oplus_{1}W^{'}$. Since $W^{'}$ is also an $L$-summand of $V$, therefore $(W^{'}, W^{'+}, \Vert.\Vert)$ is an order smooth $1$-normed space satisfying $(OS.1.2)$. We define a map $\varphi: V/W \mapsto W^{'}$ by
 $$\varphi(v+ W)= Q(v)  ~\forall~ v\in V,$$ where $Q$ is the $L$-projection of $V$ onto $W^{'}$. It is straight to check that $\varphi$ is an isometry onto $W^{'}$.
 We claim that
 $\varphi_{W}(V^+)= \overline{\varphi_{W}(V^+)}^{\Vert.\Vert}$. So let $v_{n}\in V^+$ such that $v_{n}+ W \rightarrow v+ W$ for some $v\in V$.  
 Since $Q$ is an $L$-projection, by Lemma \ref{c5}, $Q$ is a positive linear map. Thus $Q(v_{n})\in V^+$ is positive. Since $ \Vert Q(v_{n})-Q(v)\Vert =\Vert v_{n}- v+ W\Vert\rightarrow 0$, we have $Q(v)\geq 0$.
 Since $v-Q(v) \in \ker(Q)(= W)$, we have $v+ W= Q(v)+ W$. Therefore $\varphi_{W}(V^+)= \overline{\varphi_{W}(V^+)}^{\Vert.\Vert}$($=(V/W)^+$).   
 Let $v+ W\in (V/W)^+$.  Since $(V/W)^+= \varphi_{W}(V^+)$, with out loss of generality, we may assume $v\in V^+$. Then $\varphi(v+ W)=Q(v) \geq 0$ as $Q$ is an $L$-projection and $v\in V^+$. Therefore $\varphi$ is positive
 map. Conversely, let $v\in W^{+}$. Since $\varphi(v+W)= Q(v)=v$, therefore $\varphi^{-1}$ is also positive linear map.    
 Since $\varphi: V/W \mapsto W^{'}$ is an isometrical order isomorphism and $(W^{'}, W^{'+}, \Vert.\Vert)$ is an order smooth $1$-normed space satisfying $(OS.1.2)$, therefore $((V/W), (V/W)^+,  \Vert  . \Vert  )$ is an order smooth 
 $1$-normed space satisfying $(OS.1.2)$.
\end{proof}
     
    \begin{proof}[{ \bf Proof of Theorem \ref{d7}:}]  
  Let $W$ is an $M$-ideal in an order smooth $\infty$-normed space $V$. 
  Let $f\in W^{\perp}$. Since $V^{*}$ satisfies $(OS.1.2)$, there are $g, h\in V^{*+}$ such that $f=g-h$ and $ \Vert  f \Vert =  \Vert  g \Vert +  \Vert  h \Vert $. Thus by Lemma \ref{pro-of-facl-cne}, we have  $g, -h\in C(f)$.
    Since $W^{\perp}$ is an $L$-summand, 
   and $f\in W^{\perp}$, we have $g, -h\in C(f)$ so that $g, h\in W^{\perp+}$. Hence $W^{\perp}$ satisfies $(OS.1.2)$. Therefore by Theorem \ref{c3}, $(V/W, (V/W)^+,  \Vert . \Vert )$ is an order smooth $\infty$-normed space.
   \end{proof}
   
    \begin{theorem}\cite{AA18}\label{odr-cone-dec}
	Let $(V,V^+, \Vert.\Vert)$ be a complete order smooth $1$-normed space satisfying $(OS.1.2)$ and $W$ be a closed cone in $V.$ Then for any $v\in V^+,$ there are $w\in W^+$ and $w'\in W'^{+}$ 
	such that $v=w+w'$ and $\Vert v\Vert =\Vert w\Vert +\Vert w' \Vert.$
\end{theorem}
  \begin{theorem}\label{d9} Let $(V, V^+, \Vert . \Vert )$ be an order smooth $\infty$-normed space and $W$ be a subspace of $V$. If $W$ is an $M$-ideal, then following are equivalent:
   \begin{itemize}
    \item [(i)]  $(W, W^{+},  \Vert . \Vert )$  is an order smooth $\infty$-normed space;
   \item [(ii)]  if $f\in W^{*+}$, then there is a $g\in V^{*+}$ such that $g_{|_W}=f$;
  \item [(iii)] If $\varphi_{W^{\perp}}(V^{*+})= \overline{\varphi_{W^{\perp}}(V^{*+})}^{w^*}$; 
     \item [(iv)] $\Vert f\Vert= \sup\{f(w): w\in W^{+}\cap W_{1}\}$  for all  $f\in W^{*+}$.
   \end{itemize}
  \end{theorem}
  \begin{proof}
It follows from Theorem \ref{c1} that $(i)$ implies $(ii)$. Also $(ii)$ implies $(iii)$ follows from Proposition \ref{b3}. Therefore, it is enough to show the following cases: 
  \item  [$(iii)\implies (i):$]
   Since $W$ is an $M$-ideal, by Lemma \ref{c7},  $(W^{\perp'}, W^{\perp'}+, \Vert.\Vert)$ is an order smooth $1$-normed space satisfying $(OS.1.2)$. Since $\varphi_{W^{\perp}}(V^{*+})= \overline{\varphi_{W^{\perp}}(V^{*+})}^{w^*}$,
    thus by Lemma \ref{c6}, we have $(V^{*}/W^{\perp}, (V^{*}/W^{\perp})^+,  \Vert . \Vert )$ is an order smooth $1$-normed space
    satisfies $(OS.1.2)$. Hence by Theorem \ref{b2}, $(W, W^{+},  \Vert . \Vert )$ is an order smooth $\infty$-normed space. 
  \item [$(iv)\implies (ii):$]
  Let $f$ be a positive bounded linear functional on $W.$ By the Hahn Banach Theorem, there exists $g\in V^{*}$ such that $ g_{|_W}=f$ and  $\Vert g  \Vert  =   \Vert f  \Vert .$ We claim 
	that $g$ is positive. Since $V^{*}$ satisfy $(OS.1.2),$  there are $g_1, g_2 \in V^{*+}$ such that 
	$$g = g_1 - g_2  \textrm{~with~}   \Vert  g  \Vert  =  \Vert  g_1  \Vert   +  \Vert  g_2  \Vert .$$
	 Since $g_1, g_2 \in V ^{*+}$ and $V^*$ is complete, by Theorem \ref{odr-cone-dec}, there are $g_{11}, g_{21}\in W^{\perp +}$ and $g_{12}, g_{22} \in W^{\perp '+}$ such that 
	$g_1 = g_{11} + g_{12}$ with $  \Vert  g_1  \Vert  =   \Vert  g_{11}  \Vert  +   \Vert  g_{12}  \Vert $ and $g_2 = g_{21} + g_{22} $ with $  \Vert g_2  \Vert  =  \Vert g_{21}  \Vert  +   \Vert g_{22}  \Vert .$  
	Now $g = g_{11} - g_{21} + g_{12} - g_{22},$  where $g_{11}, g_{21} \in W^{\perp +}$ and $g_{12}, g_{22} \in W^{\perp'+}$ such that
	$  \Vert g  \Vert  =  \Vert g_{11}  \Vert  +   \Vert  g_{21}  \Vert  +   \Vert g_{12}  \Vert  +   \Vert  g_{22}  \Vert .$ If $f_{ij} =\left.g_{ij}\right|_W$ for all $i, j \in \{1, 2\}$. Then $f_{11} = f_{21} = 0,$ 
	so that $f = f_{12} - f_{22}.$ Further, as $f$ is positive, we have $0\leq f\leq f_{12}.$ Let $\epsilon>0$, then by assumption, there exist $w\in W^{+}\cap W_{1}$ such that $\Vert f\Vert-\epsilon< f(w). $ Since $0\leq f\leq f_{1}$,
	 thus we have $0\leq f(w)\leq f_{12}(w)$. Since $\Vert f \Vert-\epsilon \leq \Vert f_{12}\Vert$ and $\epsilon$ is arbitrary, 
	 we have $\Vert f\Vert \le \Vert f_{12}\Vert.$ Therefore,
	\begin{align*}
		  \Vert  f  \Vert & \leq    \Vert  f_{12}  \Vert \\
		& \leq    \Vert  g_{11}  \Vert  +   \Vert  g_{21}  \Vert  +   \Vert  g_{12}  \Vert  +   \Vert  g_{22}  \Vert \\
		& =    \Vert  g  \Vert  =   \Vert  f  \Vert 
	\end{align*}
	and consequently, $g_{11} = g_{21} = g_{22} = 0.$ Hence $g = g_{12} \in V^{*+}.$ 

  \item [$(i)\implies (iv):$] Let $f\in W^{*+}$. Let $\epsilon >0$, then there exist $ w\in W$ and $
 \Vert w\Vert <1$ such that $\Vert f\Vert-\epsilon < f(w)$. Since $W$ is an order smooth $\infty$-normed space, there exist $w_{1}, w_{2}\in W^{+}$ such that $w= w_{1}- w_{2}$ and $\max\{\Vert w_{1}\Vert , \Vert w_{2}\Vert \}
 < 1$ .
 Since $w_{1}, w_{2} \geq 0$, we have $f(w_{1}), f(w_{2}) \geq 0$. Now we have $\Vert f\Vert -\epsilon< f(w) \leq f(w_{1}) \leq \sup\{f(w): w\in W^{+}\cap W_{1}\}$. Since $\epsilon >0$ ia an arbitrary, we have $ \Vert f\Vert= \sup\{f(w): w\in W^{+}\cap W_{1}\}$  for all $f\in W^{*+}.$

  \end{proof}
  
  \begin{remark} \label{d10} Let $W$ be a closed subspace of $A(K)$ space.  If $W$ is an $M$-ideal in $A(K)$, then $W$ is an order smooth $\infty$-normed space (see e.g. \cite[Proposition 7.17]{LIMA72}). Hence $W$ is 
  an smooth $\infty$ order ideal of $A(K)$. In general too, we believe that these equivalent conditions are redundant in order smooth $\infty$-normed spaces, however we have not been able to show this.
  \end{remark}

{\bf Acknowledgments.}
\vskip .3cm
 I am grateful to my supervisor Dr. Anil Kumar Karn for his constant support, comments and fruitful suggestions about the paper. The work has been supported by the research grant of 
	 Department of Atomic Energy (DAE), Government of India.

\vskip .3cm
\noindent{\bf Address}:\\
{\bf Anindya Ghatak}\\
School of Mathematical Sciences\\
National Institute of Science Education and Research, Bhubaneswar, HBNI \\
P.O. - Jatni, District - Khurda, Odisha - 752050, India.\\
{\it E-mails}: anindya.ghatak@niser.ac.in

\end{document}